\documentclass[11pt]{amsproc} 

\usepackage[usenames]{color}
\usepackage{amsmath, amssymb, amsthm, setspace, mathtools, graphicx, enumitem, tikz-cd, microtype}
\usepackage{thmtools} 
\usepackage[pagebackref,colorlinks,linkcolor=blue,citecolor=blue,urlcolor=blue,hypertexnames=true]{hyperref}
\usepackage{amsrefs} 

\declaretheorem[style=plain,numberwithin=section,name=Theorem]{theorem}
\declaretheorem[style=plain,numberlike=theorem,name=Lemma]{lemma}

\declaretheorem[style=definition,numberlike=theorem,name=Definition]{definition}
\declaretheorem[style=remark,numberlike=theorem,name=Remark]{remark}
\declaretheorem[style=remark,numberlike=theorem,name=Example]{example}

\numberwithin{equation}{section}

\newcommand{\N}{\mathbb{N}}                         
\newcommand{\Z}{\mathbb{Z}}                         
\newcommand{\R}{\mathbb{R}}                         
\newcommand{\Ell}{\mathcal{L}}                      
\newcommand{\M}{\mathcal{M}}                        
\newcommand{\bnd}{\mathrm{b}}                       
\newcommand{\fin}{\mathrm{fin}}                     
\newcommand{\Ho}{\mathrm{H}\mathstrut}                        
\newcommand{\Hb}{\mathrm{H}\mathstrut_\mathrm{b}}             
\newcommand{\Hbr}{\overline{\mathrm{H}}\mathstrut_\bnd}       
\newcommand{\EHb}{\mathrm{EH}\mathstrut_\bnd}                 
\newcommand{\EHbr}{\overline{\mathrm{EH}}\mathstrut_\bnd}     
\newcommand{\Hl}{\mathrm{H}\mathstrut^{\ell^1}}               
\newcommand{\Hlr}{\overline{\mathrm{H}}\mathstrut^{\ell^1}}   
\newcommand{\Homeo}{\mathrm{Homeo}}                 
\newcommand{\Diff}{\mathrm{Diff}}                 
\newcommand{\disk}{\mathbb{D}^2}                    
\newcommand{\vol}{\mathrm{vol}}                     
\newcommand{\frgrp}{\mathrm{F}}                     
\newcommand{\ind}{\mathbf{ind}}                     
\newcommand{\id}{\mathit{id}}                       
\newcommand{\tr}{\mathbf{tr}}                       
\newcommand{\ev}{\mathit{ev}}                       
\newcommand{\one}{\mathbf{1}}                       

\DeclareMathOperator{\Ima}{Im}
\DeclareMathOperator{\Ker}{Ker}

\author{Martin Nitsche}
\address{Martin Nitsche, TU Dresden, Germany}
\email{martin.nitsche@tu-dresden.de}

\title{Higher-degree bounded cohomology of transformation groups}

\begin{document}

\onehalfspace
\setlength{\parskip}{0pt plus 1pt} 
\raggedbottom

\begin{abstract}
For $M$ a compact Riemannian manifold Brandenbursky and Marcinkowski constructed a transfer map $\Hb^*(\pi_1(M))\to\Hb^*(\Homeo_{\vol,0}(M))$ and used it to show that for certain $M$ the space $\EHbr^3(\Homeo_{\vol,0}(M))$ is infinite-dimensional. Kimura adapted the argument to $\Diff_\vol(\disk,\partial\disk)$.
We extend both results to the higher degrees $\EHbr^{2n}$, $n\geq 1$.
We also show that for certain $M$ the ordinary cohomology $\Ho^*(\Homeo_{\vol,0}(M))$ is non-trivial in all degrees.
In our computations we view the transfer map as being induced by a coupling of groups.
\end{abstract}

\maketitle
\thispagestyle{empty} 

\section{Introduction}

Let $M$ be a closed Riemannian manifold and denote by $\Homeo_{\vol,0}(M)$ those volume-preserving homeomorphisms on $M$ that are isotopic to the identity. Let $\pi_1(M)$ denote the fundamental group of $M$ and $\pi_M\vcentcolon=\pi_1(M)/Z(\pi_1(M))$ the quotient by its center. All these groups are given the discrete topology.
Building upon a construction of Gambaudo and Ghys \cite{gambaudo-ghys}, Brandenbursky and Mar\-cin\-kow\-ski \cite{brandenbursky-marcinkowski} constructed transfer maps
\begin{align*}
\Ho^*(\pi_M)&\xrightarrow{\mathmakebox[1cm]{T}}\Ho^*(\Homeo_{\vol,0}(M))\quad\text{and}\\
\Hb^*(\pi_1(M))\cong\Hb^*(\pi_M)&\xrightarrow{\mathmakebox[1cm]{T_\mathrm{b}}}\Hb^*(\Homeo_{\vol,0}(M))&
\end{align*}
on the level of ordinary and bounded cohomology.
With the help of these maps they showed that for certain fundamental groups $\pi_1(M)$ the exact reduced bounded cohomology $\EHbr^*(\Homeo_{\vol,0}(M))$ contains $\EHbr^*(\frgrp_2)$ as a subspace and, in particular, is infinite-dimensional in degree $3$.
The transfer maps $T, T_\mathrm{b}$ are induced by a certain measurable cocycle which arises from a section of the evaluation map $\ev_{x_0}\colon\Homeo_{\vol,0}(M)\to M$. It was left open whether they depend on the choice of the section.

Subsequently, Kimura \cite{kimura} used a variation of Brandenbursky and Marcinkowski's construction to obtain, among other results, transfer maps
$\Ho^*(P_m)\to\Ho^*(\Diff_\vol(\disk,\partial\disk))$ and $\Hb^*(P_m)\to\Hb^*(\Diff_\vol(\disk,\partial\disk))$,
where $P_m$ is the group of pure braids on $m$ strands. He used them to show that also $\EHbr^3(\Diff_\vol(\disk,\partial\disk))$ is infinite-dimensional.

In this article we extend the results of both Brandenbursky--Marcinkowski and Kimura to higher degrees. We approach both cases in a unified way where we view the transfer maps as being induced not by a measurable cocycle but by a coupling between groups, i.e., a measure space that carries commuting actions of the two involved groups.
Although there is a correspondence between measurable cocycles and couplings, this change in language makes the construction easier.
In particular, in both applications the coupling that induces the transfer arises naturally and does not depend on a choice.
Moreover, it becomes easier to separate that part of the construction that is specific to the particular transformation group from a more general picture where couplings between any (discrete) groups can be treated like generalized group homomorphisms and give rise to transfer maps in (bounded) cohomology.

This general picture has already been studied by various authors, although not for the setting that we require here, with a non-proper action of the uncountable discrete group $\Homeo_{\vol,0}(M)$:
Monod and Shalom \cite{monod-shalom} defined the more general induction homomorphism induced by a coupling and gave a concatenation formula, but considered only the special case where the coupling is a measure equivalence.
Savini \cite{savini} considered the homomorphisms induced by measurable cocycles in the context of lattices in Lie groups, and showed that they only depend on the cohomology class of the cocycle.
Shalom \cite{shalom} introduced the notion of a topological coupling for proper group actions and constructed the induced homomorphisms on ordinary cohomology.

In the framework of transfers induced by couplings we compute:
\newtheorem*{restate:M-bounded}{Theorem~\ref{thm:M-bounded}}
\begin{restate:M-bounded}
Let $\Homeo_{\vol,0}(M)$ denote the group of volume-preserving, isotopic-to-identity homeomorphisms on a compact Riemannian manifold $M$ of dimension $\geq 3$.
If $\pi_1(M)$ surjects onto the free group $\frgrp_2$, then $\dim\EHbr^d(\Homeo_{\vol,0}(M))=\infty$ for $d=3$ and for $d\geq 2$ even.
\end{restate:M-bounded}

\newtheorem*{restate:disk-bounded}{Theorem~\ref{thm:disk-bounded}}
\begin{restate:disk-bounded}
Let $\Diff_\vol(\disk,\partial\disk)$ denote the group of volume-preserving (smooth) diffeomorphisms of the standard $2$--disk that restrict to the identity in a neighborhood of the boundary.
Then $\dim\EHbr^d(\Diff_\vol(\disk,\partial\disk))=\infty$ for $d=3$ and for $d\geq 2$ even.
\end{restate:disk-bounded}

In \cite{loeh} Löh gave examples of groups with infinite bounded-cohomological dimension, constructed as infinite direct sums of certain base groups. With the groups $\Homeo_{\vol,0}(M)$ and $\Diff_\vol(\disk,\partial\disk)$ we now observe the same phenomenon in certain ``large'' transformation groups.

Using the same construction we also obtain the following result in ordinary cohomology:

\newtheorem*{restate:M-ordinary}{Theorem~\ref{thm:M-ordinary}}
\begin{restate:M-ordinary}
Let $M$ be a compact Riemannian manifold of dimension $\geq 5$ and assume that there exists a split epimorphism $r\colon\pi_1(M)\to\Z^2$ that is trivial on the center $Z(\pi_1(M))$.
Then $\Ho^d(\Homeo_{\vol,0}(M))\neq \{0\}$ for all $d\geq 0$.
\end{restate:M-ordinary}

\section{Preliminaries in bounded cohomology}\label{sec:preliminaries}
We quickly recall some definitions and facts about bounded cohomology.
The bar resolution $C(\Gamma;\R)$ of a discrete group $\Gamma$ is a chain complex of $\Gamma$--modules where the $n$\nobreakdash--th chain-module is the real Hilbert space with orthonormal basis $\Gamma^{n+1}$ and with the left $\Gamma$--action defined on this basis by $\gamma.(\gamma_0,\dots,\gamma_n)=(\gamma\gamma_0,\dots,\gamma\gamma_n)$.
(Some authors, like \cite{loeh}, use a slightly different notation.)
The differentials $\partial_n\colon C_n(\Gamma;\R)\to C_{n-1}(\Gamma;\R)$ are given by
\begin{align*}
\partial_n(\gamma_0,\dots,\gamma_n)=\sum_{i=0}^n (-1)^i\cdot (\gamma_0,\dots,\gamma_{i-1},\gamma_{i+1},\dots,\gamma_n).
\end{align*}

When $E$ is a coefficient module, i.e., a Banach space with an isometric $\Gamma$--action, the Banach bar complex with coefficients in $E$ is the $\ell^1$--closure of $C(\Gamma;\R)\otimes_\Gamma E$. The
$\ell^1$\nobreakdash--homology and the reduced $\ell^1$--homology of $\Gamma$ are defined as the (reduced) homology of the Banach bar chain complex:
\[
\Hl_n(\Gamma;E)\vcentcolon=\frac{\Ker\partial_n^{\ell^1}}{\Ima\partial_{n+1}^{\ell^1}},\qquad\qquad
\Hlr_n(\Gamma;E)\vcentcolon=\frac{\Ker\partial_n^{\ell^1}}{\overline{\Ima}\partial_{n+1}^{\ell^1}}
\]
Similarly, the bounded cohomology is defined by taking the topological dual of the Banach bar complex: The $n$--th module of the Banach bar cochain complex consists of the $\Gamma$\nobreakdash--invariant bounded linear functions from $C_n(\Gamma;\R)$ to $E$, identified with $\ell^\infty(\Gamma^{n+1};E)^\Gamma$, and the codifferential $\partial^n_\bnd$ is the dual of $\partial_n^{\ell^1}$.
The bounded cohomology and the reduced bounded cohomology of $\Gamma$ are defined as
\[
\Hb^n(\Gamma;E)\vcentcolon=\frac{\Ker\partial_\bnd^{n+1}}{\Ima\partial_\bnd^n},\qquad\qquad \Hbr^n(\Gamma;E)\vcentcolon=\frac{\Ker\partial_\bnd^{n+1}}{\overline{\Ima}\partial_\bnd^n}.\]
There is an inclusion map from the Banach bar cochain complex into the bar cochain complex used to define ordinary cohomology. The kernel of the induced comparison map $\mathit{cmp}\colon\Hb^*(\Gamma;E)\to\Ho^*(\Gamma;E)$ is denoted $\EHb^*(\Gamma;E)$, the exact bounded cohomology of $\Gamma$, and the image of $\EHb^*(\Gamma;E)$ under the map $\Hb^*(\Gamma;E)\to\Hbr^*(\Gamma;E)$ is denoted $\EHbr^*(\Gamma;E)$, the exact reduced bounded cohomology.

In this article we are mostly concerned with the case where $E=\R$ with the trivial group action, where we write $\Hb^*(\Gamma)\vcentcolon=\Hb^*(\Gamma;\R)$ and $\Hl_*(\Gamma)\vcentcolon=\Hl_*(\Gamma;\R)$. For this case there exists a Kronecker pairing $\langle\cdot,\cdot\rangle\colon\Hbr^n(\Gamma)\otimes\Hlr_n(\Gamma)\to\R$, just like in ordinary (co)homology. Furthermore, there is a cross product, defined on the level of cochains by
\[
(c_1\times c_2)(\gamma_0,\dots,\gamma_{p+q})=c_1(\gamma_0,\dots,\gamma_p)\cdot c_2(\gamma_p,\dots,\gamma_{p+q})
,\]
just as in ordinary cohomology, and postcomposition with the map induced by the diagonal map $\Delta\colon\Gamma\to\Gamma^2$ gives a cup product $\xi_1\cup\xi_2\vcentcolon=\Delta^*(\xi_1\times\xi_2)$.
The products are natural and compatible with the Kronecker pairing. There is, however, no K\"unneth formula.
Finally, we will use the fact that, as for ordinary cohomology, conjugation with any group element induces the identity homomorphism. Details about all of this can be found in \cite{loeh}. The proof that the cup product is graded-commutative works exactly as for ordinary cohomology \cite{hatcher}*{Theorem 3.11}.

\section{The transfer induced by a left-cofinite coupling}

We wish to write down the transfer as the composition of a generalization of Monod and Shalom's induction homomorphism with integration over the coefficient module.

\begin{definition}
By a coupling between two discrete groups $\Gamma$ and $\Lambda$ we mean a measure space $X$ together with a measure-preserving left action $\sigma$ of $\Gamma$ and a commuting measure-preserving right action $\rho$ of $\Lambda$.
We denote such a coupling either by the quintuple $(\Gamma,\sigma,X,\rho,\Lambda)$ or, if the group actions are obvious or unimportant, by the triple $(\Gamma,X,\Lambda)$.

We say that a coupling $(\Gamma,X,\Lambda)$ is \emph{left-cofinite} if $\Gamma$ is countable, the $\Gamma$--action is free, and there exists a subset $F\subset X$ of finite measure that contains exactly one point of every $\Gamma$--orbit. We call $F$ a $\Gamma$--\emph{fundamental domain}.

We say that a left-cofinite coupling $(\Gamma,X,\Lambda)$ is \emph{geometrical} if $X$ comes with the additional structure of a collection of \emph{bounded subsets} -- closed under taking subsets and finite unions -- such that
both group actions preserve boundedness, for every $B\subset X$ bounded the set $\{\gamma\in\Gamma\mid\gamma B\cap B\neq\emptyset\}$ is finite, and the fundamental domain can be chosen to be bounded.
\end{definition}

\begin{example}\label{exm:main-example}
Let $M$ be a compact topological manifold that is also a finite measure space. The fundamental group $\Gamma\vcentcolon=\pi_1(M)$ acts freely and properly on the universal covering $X\vcentcolon=\widetilde{M}$.
Let the bounded subsets of $X$ be exactly the precompact subsets.
If the group $\Lambda$ acts on $X$ by measure-preserving continuous maps, then $(\Gamma,X,\Lambda)$ is a geometrical left-cofinite coupling. Note that $\Lambda$ may be uncountable or act non-properly.
\end{example}

When $(\Gamma,\sigma,X,\rho,\Lambda)$ is a left-cofinite coupling, a choice of a fundamental domain $F\subset X$ gives rise to a $\Gamma$--equivariant isomorphism of measure spaces $X\cong\Gamma\times F$, which preserves boundedness if the coupling is geometrical.
From this isomorphism one obtains a retract $\chi\colon X\to\Gamma$ that sends $\gamma.y$ to $\gamma$ for all $y\in F$, $\gamma\in\Gamma$, and hence satisfies $\chi(\gamma.x)=\gamma\cdot\chi(x)$.
(One also gets an action $\Lambda\curvearrowright F$ and a map $\chi\circ\rho\colon\Lambda\times F\to\Gamma$. Together they form the \emph{associated measurable cocycle}, which contains the same data as the coupling plus the choice of $F$.)

For a given left-cofinite coupling $(\Gamma,X,\Lambda)$ and a $\Gamma$--coefficient module $E$ let $\Ell^\infty(X,E)$ denote the Lebesgue space of essentially bounded functions on $X$. It carries isometric $\Gamma$-- and $\Lambda$--actions given by
$(\gamma.f)(x)\vcentcolon=\gamma.f(\gamma^{-1}.x)$ and $(\lambda.f)(x)\vcentcolon=f(x.\lambda)$.
Recall that in the definition of the Banach bar cochain complex the group action on $\ell^\infty(\Gamma^{n+1};E)$ is given by
$(\gamma.c)(\gamma_0,\dots,\gamma_n)\vcentcolon=\gamma.c(\gamma^{-1}\gamma_0,\dots,\gamma^{-1}\gamma_n)$.
In parallel to \cite{monod-shalom}*{Proposition 4.5}, we define for every $n\in\N$ a map
\begin{equation}\label{eqn:monod-shalom-map}
\begin{gathered}
\chi^*\colon\ell^\infty(\Gamma^{n+1};E)^\Gamma\to\ell^\infty\big(\Lambda^{n+1};\Ell^\infty(X,E)^\Gamma\big)^\Lambda\\
\chi^*c(\lambda_0,\dots,\lambda_n)(x)\vcentcolon=c\big(\chi(x.\lambda_0),\dots,\chi(x.\lambda_n)\big).
\end{gathered}
\end{equation}
It is easy to check that $\chi^*$ is a well-defined bounded homomorphism of cochain complexes and hence induces an induction homomorphism
on bounded cohomology and on reduced bounded cohomology:
\[\ind_\Gamma^\Lambda X\colon\qquad\Hb^*(\Gamma;E)\to\Hb^*(\Lambda;\Ell^\infty(X,E)^\Gamma),\qquad\Hbr^*(\Gamma;E)\to\Hbr^*(\Lambda;\Ell^\infty(X,E)^\Gamma)
\]
If the coupling is geometrical, then $\chi$ has a finite image when restricted to bounded subsets. In this case we define an induction homomorphism on ordinary cohomology by the same formula, just replacing $\Ell^\infty(X,E)$ with $\M_\fin(X,E)$, the space of measurable functions with finite image on bounded sets.
Then we have $\mathit{cmp}\circ\ind_\Gamma^\Lambda X=\iota_*\circ\ind_\Gamma^\Lambda X\circ\mathit{cmp}$, where $\iota\colon\M_\fin(X,E)\to\Ell^\infty(X,E)$ denotes the inclusion, and hence the induction is also defined on exact reduced bounded cohomology:
\[\ind_\Gamma^\Lambda X\colon\qquad\Ho^*(\Gamma;E)\to\Ho^*(\Lambda;\M_\fin(X,E)^\Gamma),\qquad\EHbr^*(\Gamma;E)\to\EHbr^*(\Lambda;\Ell^\infty(X,E)^\Gamma)
\]

\begin{lemma}
The induction homomorphism does not depend on the choice of the fundamental domain.
\end{lemma}
\begin{proof}
Assume that $F_1, F_2$ are two fundamental domains of $X$ with associated retracts $\chi_1, \chi_2$.
Let $\Lambda_0$ denote the group of $\Gamma$--equivariant measure-preserving (and, in the geometrical case, boundedness-preserving) automorphisms on $X$. Since $\chi_1^*$ and $\chi_2^*$ factor through the induced map of the group homomorphism $\Lambda\to\Lambda_0$ given by the action $\rho$, we may assume $\Lambda=\Lambda_0$.

The isomorphism of measure spaces $F_1\to\Gamma\times F_1\cong X\cong\Gamma\times F_2\to F_2$ gives rise to a map $g\colon X\cong\Gamma\times F_1\to\Gamma\times F_2\cong X$ such that $\chi_1=\chi_2\circ g$. But $g$ is an element of $\Lambda_0$. Hence, if $r_g\colon\Lambda_0\to\Lambda_0$ denotes right-multiplication with $g$, then $\chi_1^*=(\chi_2\circ r_g)^*=r_g^*\circ\chi_2^*$. Finally, in both ordinary and bounded cohomology $r_g$ induces the identity by the cochain homotopy
\[\Theta_{n+1}\colon C^{n+1}_{(\mathrm{b})}\to C^{n}_{(\mathrm{b})},\quad
\Theta_{n+1}(c)(\lambda_0,\dots,\lambda_{n})\vcentcolon=\sum_{i=0}^{n}(-1)^i\cdot c(\lambda_0,\dots,\lambda_i,\lambda_i g,\dots,\lambda_n g).\qedhere
\]
\end{proof}

In the case where $E=\R$ with the trivial $\Gamma$--action we now postcompose the induction with the map induced by integrating the coefficients over any $\Gamma$--fundamental domain $F$ -- using the fact that the $\Lambda$--action is measure-preserving. Thus, we recover the transfer map:
\begin{equation}\label{eqn:transfer-formula}
\begin{gathered}
\tr_\Gamma^\Lambda X\colon\qquad\Hb^*(\Gamma)\to\Hb^*(\Lambda),\qquad\Hbr^*(\Gamma)\to\Hbr^*(\Lambda)\\
[c]\mapsto\left[(\lambda_0,\dots,\lambda_n)\mapsto\int_F c\big(\chi(x.\lambda_0),\dots,\chi(x.\lambda_n)\big)\right]
\end{gathered}
\end{equation}
If the coupling is geometrical, the same formula again works for ordinary and exact reduced bounded cohomology, and we have $\mathit{cmp}\circ\tr_\Gamma^\Lambda=\tr_\Gamma^\Lambda\circ\mathit{cmp}$.

\begin{remark}
Equation \eqref{eqn:transfer-formula} is, essentially, the same formula that appears in \cite{brandenbursky-marcinkowski}, as well as in \cite{savini} and other places.
Minor differences in appearance are the result of different conventions in the definition of (bounded) cohomology. The reason why a coupling $(\Gamma,X,\Lambda)$ induces a map $\Hb^*(\Gamma)\to\Hb^*(\Lambda)$ instead of a map $\Hb^*(\Lambda)\to\Hb^*(\Gamma)$, as one might expect, is that we stick with the usual group action for the bar resolution, $\gamma.(\gamma_0,\dots,\gamma_n)=(\gamma\gamma_0,\dots,\gamma\gamma_n)$, instead of $\gamma.(\gamma_0,\dots,\gamma_n)=(\gamma_0\gamma^{-1},\dots,\gamma_n\gamma^{-1})$.
These kinds of issues can be overcome by composing the transfer with the map induced by the group involution, $\cdot^{-1}\colon\Gamma\to\Gamma^\mathrm{op}$.
\end{remark}

\begin{lemma}\label{lem:transfer-properties}
The transfer construction satisfies the following properties:
\begin{enumerate}
\item\label{itm:morphisms}
If $X=\Gamma$ with the counting measure and with $\Gamma$ acting by left multiplication, then the $\Lambda$--action in any coupling $(\Gamma,X,\Lambda)$ is given by a homomorphism $\varphi\colon\Lambda\to\Gamma$, $x.\lambda=x\cdot\varphi(\lambda)$, and we have
\[\qquad\qquad\tr_\Gamma^\Lambda X=\varphi^*.\]
\item\label{itm:sums-and-products}
If $X_\mathrm{tr}$ and $X'_\mathrm{tr}$ are finite (bounded) measure spaces with trivial $\Gamma$-- and $\Lambda$--actions, then for any (geometrical) left-cofinite couplings $(\Gamma,X,\Lambda)$, $(\Gamma,X',\Lambda)$ we have
\[\qquad\qquad\tr_\Gamma^\Lambda(X\times X_\mathrm{tr}\sqcup X'\times X'_\mathrm{tr})=\vol(X_\mathrm{tr})\cdot\tr_\Gamma^\Lambda X+\vol(X'_\mathrm{tr})\cdot\tr_\Gamma^\Lambda X'.\]
\item\label{itm:concatenation}
If $(\Gamma,X_1,\Lambda)$ and $(\Lambda,X_2,\Pi)$ are (geometrical) left-cofinite couplings, and $X_1\times_\Lambda X_2$ denotes the quotient of $X_1\times X_2$ by the $\Lambda$--action $\lambda.(x_1,x_2)=(x_1.\lambda^{-1},\lambda.x_2)$, then
$(\Gamma,X_1\times_\Lambda X_2,\Pi)$ is a (geometrical) left-cofinite coupling and \[\qquad\qquad
\tr_\Gamma^\Pi(X_1\times_\Lambda X_2)=\tr_\Lambda^\Pi X_2\circ\tr_\Gamma^\Lambda X_1.\]
Furthermore, if $(\Gamma,\sigma,X,\rho,\Lambda)$ is a left-cofinite coupling and $\varphi\colon\Pi\to\Lambda$ is a homomorphism between the possibly uncountable groups, then
\[\qquad\qquad
\tr_{\Gamma,\sigma}^{\Pi,\rho\circ\varphi} X=\varphi^*\circ\tr_{\Gamma,\sigma}^{\Lambda,\rho}X.\]
\item\label{itm:cross-product}
If $(\Gamma_1,X_1,\Lambda_1)$ and $(\Gamma_2,X_2,\Lambda_2)$ are two (geometrical) left-cofinite couplings, then $(\Gamma_1\times\Gamma_2,X_1\times X_2,\Lambda_1\times\Lambda_2)$ is a (geometrical) left-cofinite coupling and
\[\qquad\qquad
\tr_{\Gamma_1\times\Gamma_2}^{\Lambda_1\times\Lambda_2} (X_1\times X_2)(\xi_1\times\xi_2)=
\tr_{\Gamma_1}^{\Lambda_1} X_1(\xi_1)\times
\tr_{\Gamma_2}^{\Lambda_2} X_2(\xi_2).\]
\item\label{itm:limits}
If $\big((\Gamma,\sigma,X,\rho_l,\Lambda)\big)_{l\in\N}$ is a sequence of left-cofinite couplings and the $\rho_l$ converge to an action $\rho_\infty$ in the sense that $\vol\big(F\cap\bigcup_{\lambda\in\Lambda}\{\rho_l(\lambda)(x)\neq\rho_\infty(\lambda)(x)\}\big)\to 0$ for any fundamental domain $F$, then the transfer maps in reduced bounded cohomology converge in the sense that
\[\qquad\qquad\sup_{\|\xi\|=1}\big\|\tr_{\Gamma,\sigma}^{\Lambda,\rho_l}X(\xi)-\tr_{\Gamma,\sigma}^{\Lambda,\rho_\infty}X(\xi)\big\|\to 0.\]
If, in addition, all couplings are geometrical and the sequence is uniformly bounded in the sense that
\begin{equation}\tag{UB}\label{eqn:bounded-convergence}
\qquad
\text{for any fundamental domain $F$ and all }\lambda\in\Lambda\colon
\bigcup_{l\in\N}\rho_l(\lambda)(F) \text{ is bounded},
\end{equation}
then the transfer maps in ordinary cohomology converge in the sense that for all $\xi\in\Ho^n(\Gamma)$ and all $\xi'\in\Ho_n(\Lambda)$
\[\qquad\qquad\big\langle\tr_{\Gamma,\sigma}^{\Lambda,\rho_l}X(\xi),\xi'\big\rangle\to\big\langle\tr_{\Gamma,\sigma}^{\Lambda,\rho_\infty}X(\xi),\xi'\big\rangle.\]
\end{enumerate}
\end{lemma}
\begin{proof}
Properties \ref{itm:morphisms}, \ref{itm:sums-and-products} and the first part of \ref{itm:limits} follow directly from Equation~\eqref{eqn:transfer-formula}. In the second part of property~\ref{itm:limits} the condition \eqref{eqn:bounded-convergence} ensures that for all tuples $(\lambda_0,\dots,\lambda_n)$ the functions that appear under the integral, $x\mapsto c\big(\chi\big(\rho_l(\lambda_0)(x)\big),\dots,\chi\big(\rho_l(\lambda_n)(x)\big)\big)$, are uniformly bounded over $x$ and $l$.

For property~\ref{itm:concatenation}, let $q\colon X_1\times X_2\to X_1\times_\Lambda X_2$ be the quotient map. The measure on $X_1\times_\Lambda X_2$ is defined by $\vol(A)=\vol(q^{-1}(A)\cap X_1\times F_2)$ for any $\Lambda$--fundamental domain $F_2\subset X_2$. In the geometrical case a subset $A\subset X_1\times_\Lambda X_2$ is defined to be bounded iff $q^{-1}(A)\cap X_1\times F_2$ is bounded. This makes $(\Gamma,X_1\times_\Lambda X_2,\Pi)$ a (geometrical) left-cofinite coupling.

If $F_1\subset X_1$ and $F_2\subset X_2$ are fundamental domains for the $\Gamma$--, respectively the $\Lambda$--action, then $q$ restricted to $F_1\times F_2$ is injective and $F\vcentcolon=q(F_1\times F_2)$ is a $\Gamma$--fundamental domain on $X_1\times_\Lambda X_2$.
If $\chi_1$ and $\chi_2$ are the retracts associated to $F_1$ and $F_2$, then the map $\chi\colon X_1\times_\Lambda X_2\to\Gamma$ sending $q(x_1,x_2)$ to $\chi_1(x_1.\chi_2(x_2))$ is well-defined, compatible with the $\Gamma$--action, and sends $F$ to $\one\in\Gamma$. Hence, it is the retract associated to $F$. Now we compute on the level of cochains:
\begin{align*}
&\big(\tr_\Lambda^\Pi X_2\circ\tr_\Gamma^\Lambda X_1\big)(c)(\pi_0,\dots,\pi_n)\\
&\qquad=\int_{F_2}\int_{F_1} c\big(\chi_1(x_1.\chi_2(x_2.\pi_0)),\dots,\chi_1(x_1.\chi_2(x_2.\pi_n)\big)\\
&\qquad=\int_{F_1\times F_2} c\Big(\chi\big(q((x_1,x_2).\pi_0)\big),\dots,\chi\big(q((x_1,x_2).\pi_n)\big)\Big)\\
&\qquad=\int_F c\big(\chi(x.\pi_0),\dots,\chi(x.\pi_n)\big)\\
&\qquad=\tr_\Gamma^\Pi(X_1\times_\Lambda X_2)(c)(\pi_0,\dots,\pi_n).
\end{align*}
The second part of the statement follows directly from Equation~\eqref{eqn:transfer-formula}.

For property~\ref{itm:cross-product}, let $F_{1/2}$ be $\Gamma_{1/2}$--fundamental domains and $\chi_{1/2}$ the associated retracts. Then $F_1\times F_2$ is a $\Gamma_1\times\Gamma_2$--fundamental domain with associated retract $\chi_1\times\chi_2$. If $\xi_{1/2}$ are represented by cochains $c_{1/2}$ of degree $p$ and $q$, respectively, then we compute
\begin{align*}
&\tr_{\Gamma_1\times\Gamma_2}^{\Lambda_1\times\Lambda_2}(X_1\times X_2)(c_1\times c_2)(\lambda_0,\dots,\lambda_{p+q})\\
&\qquad=\int_{F_1\times F_2} (c_1\times c_2)\big((\chi_1\times\chi_2)((x_1,x_2).\lambda_0),\dots,(\chi_1\times\chi_2)((x_1,x_2).\lambda_{p+q})\big)\\
&\qquad=\int_{F_1\times F_2} c_1\big(\chi_1(x_1.\lambda_0),\dots,\chi_1(x_1.\lambda_p)\big)\cdot c_2\big(\chi_2(x_2.\lambda_p),\dots,\chi_2(x_2.\lambda_{p+q})\big)\\
&\qquad=\tr_{\Gamma_1}^{\Lambda_1}X_1(c_1)(\lambda_0,\dots,\lambda_p)\cdot\tr_{\Gamma_2}^{\Lambda_2}X_2(c_2)(\lambda_p,\dots,\lambda_{p+q})\\
&\qquad=\big(\tr_{\Gamma_1}^{\Lambda_1}X_1(c_1)\times\tr_{\Gamma_2}^{\Lambda_2}X_2(c_2)\big)(\lambda_0,\dots,\lambda_{p+q}).\qedhere
\end{align*}
\end{proof}

\section{Proof of the main results}

We now prove our lower bounds for the bounded and ordinary cohomology of the groups $\Homeo_{\vol,0}(M)$ and $\Diff_\vol(\disk,\partial\disk)$.
In all cases the proof strategy is to meet the requirements of the following lemma.
While its part~\ref{itm:direct-calculation} is directly taken from \cite{brandenbursky-marcinkowski}, our new results in degree $>3$ rely on part~\ref{itm:pairing-calculation}.
The trick of applying the Kronecker pairing to compensate for the lack of a Künneth formula in bounded cohomology was previously used in \cite{loeh}.

\begin{lemma}\label{lem:higher-degrees}
Let $\Gamma, \Lambda, G$ be discrete groups and $n\in\N_{>0}$.
Assume that:
\begin{enumerate}
\item\label{itm:coupling-step}
There exists a left-cofinite geometric coupling $(\Gamma,\sigma,X,\rho,\Lambda)$ between $\Gamma$ and $\Lambda$.
\item\label{itm:sequence-step}
There exist $n$--many sequences of group homomorphisms
$\big(\alpha_{i,l}\colon G\to\Lambda\big)_{i\leq n,l\in\N}$, such that for $l$ fixed the images of the homomorphisms $\{\alpha_{i,l}\}_{i\leq n}$ commute.
Furthermore, for each $i$ the sequence of actions $(\rho\circ\alpha_{i,l})_{l\in\N}$ converges, in the sense of Lemma~\ref{lem:transfer-properties}, property~\ref{itm:limits}, to an action $\beta_i$.
\item\label{itm:compute-step}
The sets of support $Y_i\vcentcolon=\{x\in X\mid\exists g\in G\colon\beta_i(g)(x)\neq x\}$ are disjoint and in degree $>0$ all the transfers $\tr_{\Gamma,\sigma}^{G,\beta_i}Y_i$ give the same map $\phi\colon\Hbr^*(\Gamma)\to\Hbr^*(G)$.
\end{enumerate}
Then it follows:
\begin{enumerate}[label=(\roman*), ref=\roman*]
\item\label{itm:direct-calculation}
$\dim\Hbr^d(\Lambda)\geq\dim\phi\big(\Hbr^d(\Gamma)\big)$, and
\item\label{itm:pairing-calculation}
$\dim\Hbr^d(\Lambda)\geq\dim\mathrm{GrSym}^n_d\big(\mathit{Kr}\circ\phi\big(\Hbr^*(\Gamma)\big)\big)$,
where $\mathit{Kr}\colon\Hbr^*(G)\to\big(\Ho^{\ell^1}_*(G)\big)'$ is the map induced by the Kronecker pairing, where $\mathrm{GrSym}^n_d$ is the degree--$d$ part of the $n$--fold graded-symmetric algebraic tensor product for graded vector spaces, and where we do not distinguish between infinite dimensions of different cardinalities.
\end{enumerate}
The analog results hold for the exact reduced bounded cohomology groups, replacing all $\Hbr$ with $\EHbr$.
If the sequences $(\rho\circ\alpha_{i,l})_{l\in\N}$ satisfy condition \eqref{eqn:bounded-convergence}, then the analog results hold also for ordinary cohomology.
\end{lemma}
\begin{proof}
We focus on the reduced bounded cohomology. The cases for exact reduced bounded cohomology and ordinary cohomology are completely analogous.

\textbf{i)} We may assume $n=1$ and $d>0$. Let $V\subset\Hbr^d(\Gamma)$ be a linear subspace such that $\dim V=\dim\phi\big(\Hbr^d(\Gamma)\big)$ and such that $\phi$ is injective on $V$. For any $v\neq 0\in V$ we have $\tr_{\Gamma,\sigma}^{G,\beta_1}X(v)=\phi(v)\neq 0$ and it follows from Lemma
\ref{lem:transfer-properties}, property~\ref{itm:limits}
that there exists an $l$ such that $0\neq\tr_{\Gamma,\sigma}^{G,\rho\circ\alpha_{1,l}}X(v)=(\alpha_{1,l})^*\circ\tr_{\Gamma,\sigma}^{\Lambda,\rho}X(v)$. Hence, $\tr_{\Gamma,\sigma}^{\Lambda,\rho}X$ is injective on $V$ and the statement follows.

\textbf{ii)}
Let $\Delta_\Lambda\colon\Lambda\to\Lambda^n$ be the diagonal homomorphism.
From the given data we construct the new left-cofinite coupling $(\Gamma^n,\sigma^n,X^n,\rho^n\circ\Delta_\Lambda,\Lambda)$ and the new sequence of group homomorphisms $(\alpha_l)_{l\in\N}$ where $\alpha_l\vcentcolon=\prod_i\alpha_{i,l}\colon G^n\to\Lambda$.
Because the supports $Y_i$ are disjoint, we have
\[\bigcup_{\bar{g}\in G^n}\Big\{x\in X\Bigm|\rho\circ\alpha_{l}(\bar{g})(x)\neq\big({\textstyle\prod_i}\beta_i\big)(\bar{g})(x)\Big\}
\subset\bigcup_{i\leq n}\bigcup_{g\in G}\Big\{x\in X\Bigm|\rho\circ\alpha_{i,l}(g)(x)\neq\beta_i(g)(x)\Big\},\]
and it follows that
the sequence of group actions $(\rho\circ\alpha_l)_{l\in\N}$ converges to $\prod_i\beta_i$.
If all sequences $(\rho\circ\alpha_{i,l})_{l\in\N}$ satisfy condition \eqref{eqn:bounded-convergence}, then the same is true for $(\rho\circ\alpha_l)_{l\in\N}$ because it suffices to check this condition on the generators of $G^n$.

From here, the sequence of product group actions $\big((\rho\circ\alpha_l)^n\big)_{l\in\N}$ of $(G^n)^n$ on $X^n$ converges to $(\prod_i\beta_i)^n$, and the group actions $(\rho^n\circ\Delta_\Lambda\circ\alpha_l)_{l\in\N}=\big((\rho\circ\alpha_l)^n\circ\Delta_{G^n}\big)_{l\in\N}$ converge to $\beta\vcentcolon=(\prod_i\beta_i)^n\circ\Delta_{G^n}$.
To prove the statement we will apply part~\ref{itm:direct-calculation} to the coupling $(\Gamma^n,\sigma^n,X^n,\rho^n\circ\Delta_\Lambda,\Lambda)$ and the single sequence of group homomorphisms $(\alpha_l)_{l\in\N}$. It remains to show that $\dim\tr_{\Gamma^n,\sigma^n}^{G^n,\beta}X^n\big(\Hbr^d(\Gamma^n)\big)\geq\dim\mathrm{GrSym}^n_d\big(\mathit{Kr}\circ\phi\big(\Hbr^*(\Gamma)\big)\big)$.

Using Lemma~\ref{lem:transfer-properties}, properties \ref{itm:concatenation} and \ref{itm:cross-product}, we compute for all $\xi_1,\dots,\xi_n\in\Hbr^*(\Gamma)$:
\begin{align*}
\tr_{\Gamma^n,\sigma^n}^{G^n,\beta}X^n(\xi_1\times\ldots\times\xi_n)
&=(\Delta_{G^n})^*\circ\tr_{\Gamma^n,\sigma^n}^{(G^n)^n,(\Pi_i\beta_i)^n}X^n(\xi_1\times\ldots\times\xi_n)\\
&=(\Delta_{G^n})^*\Big(\tr_{\Gamma,\sigma}^{G^n,\Pi_i\beta_i}X(\xi_1)\times\ldots\times\tr_{\Gamma,\sigma}^{G^n,\Pi_i\beta_i}X(\xi_n)\Big)\\
&=\tr_{\Gamma,\sigma}^{G^n,\Pi_i\beta_i}X(\xi_1)\cup\ldots\cup\tr_{\Gamma,\sigma}^{G^n,\Pi_i\beta_i}X(\xi_n).
\end{align*}
For $\xi_j\in\Hbr^{>0}(\Gamma)$ the cup product factors can be simplified further by decomposing $X$ as the disjoint union of the supports $\{Y_j\}_{j\leq n}$ and some inconsequential remainder $Y_R$.
Letting $q_j\colon G^n\to G$ denote the projection onto the $j$--th factor, we obtain from Lemma~\ref{lem:transfer-properties}, property~\ref{itm:sums-and-products}:
\begin{align*}
\tr_{\Gamma,\sigma}^{G^n,\Pi_i\beta_i}X(\xi)
&=\sum_{j\leq n}\tr_{\Gamma,\sigma}^{G^n,\Pi_i\beta_i}Y_{j}(\xi)
=\sum_{j\leq n}\tr_{\Gamma,\sigma}^{G^n,\beta_{j}\circ q_{j}}Y_{j}(\xi)\\
&=\sum_{j\leq n}{q_{j}}^*\circ\tr_{\Gamma,\sigma}^{G,\beta_{j}}Y_{j}(\xi)
=\sum_{j\leq n}{q_{j}}^*\circ\phi(\xi).
\end{align*}
On $\Hbr^0$, meanwhile, $\tr_{\Gamma,\sigma}^{G^n,\Pi_i\beta_i}X$ is just multiplication by a positive scalar. For convenience we define $\phi$ also on $\Hbr^0$ in such a way that the preceding equality holds.

For each degree $d$ we find subsets $\{z_k\}_{k\in K_d}\subset\Hbr^d(\Gamma)$ and $\{z'_k\}_{k\in K_d}\subset\Ho^{\ell^1}_d(G)$ such that $\langle\phi(z_k),z'_{k'}\rangle=\delta_{k=k'}$ and such that $|K_d|=\dim\mathit{Kr}\circ\phi\big(\Hbr^d(\Gamma)\big)$ (in the infinite-dimensional case we can find $K_d$ countably infinite).
Let $0\in K_0$ denote the unique index in degree~$0$. (In the analog proof for exact reduced bounded cohomology we have $\EHbr^0(\Gamma)=\{0\}$ and avoid the complications arising from this element.)
We fix a total order on $K\vcentcolon=\bigsqcup_d K_d$ and let $\bar{K}\subset K^n$ denote those $n$--tuples that are ordered and may contain duplicate entries only of indices $k$ where the degree $|z_k|$ is even.

Assume now that $\bar{k}=(k_1,\dots,k_n)$ and $\bar{k}'=(k'_1,\dots,k'_n)$ are two tuples in $\bar{K}$.
For any partition function
$P\in\mathcal{P}\vcentcolon=\big\{\{1,\dots,n\}\to\{1,\dots,n\}\big\}$ we let $S_P$
be a permutation such that $P\circ S_P$ is monotonously increasing, we denote by $\varepsilon_P$ the sign that arises from the graded-commutative reordering of $(z_{k_1},\dots,z_{k_n})$ to $(z_{k_{S_P(1)}},\dots,z_{k_{S_P(n)}})$, and we denote by $L_P^{(i)}$ the ordered subset $(P\circ S_P)^{-1}(i)\subset\{1,\dots,n\}$.
The Kronecker pairing of
\begin{align*}
\tr_{\Gamma^n,\sigma^n}^{G^n,\beta}X^n(z_{k_1}\times\ldots\times z_{k_n})
&=\bigcup_{i=1}^n\sum_{j=1}^n{q_j}^*\circ \phi(z_{k_i})
=\sum_{P\in\mathcal{P}}\bigcup_{i=1}^n q_{P(i)}^*\circ\phi(z_{k_i})\\
&=\sum_{P\in\mathcal{P}}\varepsilon_P\cdot\bigcup_{i=1}^n q_{P\circ S_P(i)}^*\circ\phi(z_{k_{S_P(i)}})\\
&=\sum_{P\in\mathcal{P}}\varepsilon_P\cdot \Big({\textstyle\bigcup_{i\in L_P^{(1)}}}\phi(z_{k_i})\Big)\times\dots\times\Big({\textstyle\bigcup_{i\in L_P^{(n)}}}\phi(z_{k_i})\Big)
\end{align*}
with $z'_{k'_1}\times\dots\times z'_{k'_n}$ is given by the sum
$\sum_{P\in\mathcal{P}}\varepsilon_P\cdot\prod_{j=1}^n\big\langle{\textstyle\bigcup_{i\in K_{S_P}^{(j)}}}\phi(z_{k_i}),z'_{k'_j}\big\rangle$.
If the multiplicity of $0$ is higher in $\bar{k}$ than in $\bar{k}'$, then all its summands are zero, and the same happens if the multiplicities are the same but $\bar{k}\neq\bar{k}'$. If $\bar{k}=\bar{k}'$, however, there appear non-zero summands and they all have the same sign $\varepsilon_P=1$.

Now, if $w$ is a non-trivial linear combination over $W=\big\{\tr_{\Gamma^n,\sigma^n}^{G^n,\beta}X^n(z_{k_1}\times\ldots\times z_{k_n})\big\}_{\bar{k}\in\bar{K}}$
and $(k_1,\dots,k_n)$ describes an element in the support of $w$ with a minimal multiplicity of $0\in K$, then
$\langle w,z'_{k_1}\times\dots\times z'_{k_n}\rangle\neq 0$.
Consequently, the classes in $W$ are linearly independent and $\dim\tr_{\Gamma^n,\sigma^n}^{G^n,\beta}X^n\big(\Hbr^d(\Gamma^n)\big)\geq |W| =\dim\mathrm{GrSym}^n_d\big(\mathit{Kr}\circ\phi\big(\Hbr^*(\Gamma)\big)\big)$.
\end{proof}

To obtain the coupling and the sequences of group homomorphisms required by the previous lemma, we follow the constructions of Brandenbursky--Marcinkowski \cite{brandenbursky-marcinkowski}, respectively Kimura \cite{kimura}, with only some slight modifications. For the convenience of the reader we present the full argument, adapted to the setting of couplings between groups.

\begin{theorem}\label{thm:M-bounded}
Let $M$ be a compact Riemannian manifold of dimension $\geq 3$ and assume that $\pi_1(M)$ surjects onto the free group $\frgrp_2$.
Then $\dim\EHbr^d(\Homeo_{\vol,0}(M))=\infty$ for $d=3$ and for $d\geq 2$ even.
\end{theorem}
\begin{proof}
Let $\Gamma\vcentcolon=\pi_1(M)/Z(\pi_1(M))$ be the quotient of the fundamental group of $M$ by its center, and let $\Lambda\vcentcolon=\Homeo_{\vol,0}(M)$ with the multiplication $\lambda_1\lambda_2=\lambda_2\circ\lambda_1$. For each degree $d$ we also fix some $n\in\N$ such that $2n\geq d$.

\textbf{1)}
Let the measure space of the required coupling $(\Gamma,\sigma,X,\rho,\Lambda)$ be the covering space
$p\colon X\to M$ associated to the subgroup $Z(\pi_1(M))<\pi_1(M)$. Its measure is induced from $M$ and we take the precompact subsets of $X$ as the bounded sets, as in Example~\ref{exm:main-example}.
There is a left action $\sigma\colon\Gamma\curvearrowright X$ of deck transformations.
Let the right action $\rho\colon X\curvearrowleft\Lambda$ be defined via monodromy by $x.\lambda\vcentcolon=\bar{H}_1(x)$, where $\bar{H}\colon [0,1]\times X\to X$ is the lift of any isotopy $H$ from the identity on $M$ to $\lambda$.

We check that $\rho$ is well-defined: If $H'$ is a second isotopy from the identity to $\lambda$ and $\widetilde{H},\widetilde{H}'$ are the lifts of $H,H'$ to the universal covering, then $\widetilde{H}'_1\circ(\widetilde{H}_1)^{-1}$ is a deck transformation and hence agrees with the action of some $\gamma\in\pi_1(M)$. Since $\widetilde{H}'_1\circ(\widetilde{H}_1)^{-1}$ is $\pi_1(M)$--equivariant, $\gamma\in Z(\pi_1(M))$. Descending to $X$, we get $\bar{H}'_1\circ(\bar{H}_1)^{-1}=\id$.

\textbf{2)}
Let $G\vcentcolon=\frgrp_2=\langle a_1,a_2\rangle$ be the free group over two generators. Let $s\colon\frgrp_2\to\pi_1(M)$ be a split of a surjection $r\colon\pi_1(M)\to\frgrp_2$, and let $\{\gamma_{i,1},\gamma_{i,2}\}_{i\leq n}$ be $n$--many copies of two loops in $M$ that represent $s(a_{1})$, respectively $s(a_2)$.
Since $\dim M\geq 3$, we may perturb the loops slightly to smooth embeddings of the circle, such that the pairs $\gamma_{i,1},\gamma_{i,2}$ intersect and all other intersections are empty. We take tubular neighborhoods $N(\gamma_{i,1/2})$ of the loops such that the only non-empty intersections are $N(\gamma_{i,1})\cap N(\gamma_{i,2})$ and these all have the same volume $\kappa$.
Using the relative version of Moser's trick \cite{banyaga} we identify each $N(\gamma_{i,1/2})$, by a volume-preserving homeomorphism, with $S^1\times B_\epsilon$, the product of the standard circle with a ball in $(\dim M-1)$--dimensional Euclidean space.

We define the sequences of homomorphisms $\big(\alpha_{i,l}\colon G=\frgrp_2\to\Lambda\big)_{i\leq n,l\in\N}$ by the actions of the two generators: 
For $j\in\{1,2\}$ we let the map $\alpha_{i,l}(a_j)\colon X\to X$ be the identity outside of $N(\gamma_{i,j})$ for all $l\in\N$. Inside $N(\gamma_{i,j})\cong S^1\times B_\epsilon$ we let it be the end result of the ``finger-push'' homotopy that fixes the boundary, pushes $S^1\times B_{\epsilon-\epsilon/l}$ once around the circle and interpolates continuously on the rest.
The sequence $(\rho\circ\alpha_{i,l})_{l\in\N}$ converges, in the sense of Lemma~\ref{lem:transfer-properties}, property~\ref{itm:limits}, to the action $\beta_i$, where $\beta_i(a_j)$ is induced by rotating the whole subset $N(\gamma_{i,j})$ and fixing its complement.

\textbf{3)}
For each $i$ the limit action $\beta_i$ has support
$Y_i\vcentcolon=p^{-1}\big(N(\gamma_{i,1})\cup N(\gamma_{i,2})\big)\subset X$.
The coupling $(\Gamma,\sigma,Y_i,\beta_i,G)$ decomposes into a disjoint union of products:
\begin{equation*}\begin{gathered}
Y_i=\,\, p^{-1}(X_{i,0})\sqcup p^{-1}(X_{i,1})\sqcup p^{-1}(X_{i,2})\,\,\cong\,\,\Gamma\times X_{i,0}\sqcup\Gamma\times X_{i,1}\sqcup\Gamma\times X_{i,2},\\[-0.0cm]
\text{where}\quad X_{i,0}\vcentcolon=N(\gamma_{i,1})\cap N(\gamma_{i,2}),\quad X_{i,1/2}\vcentcolon=N(\gamma_{i,1/2})\setminus N(\gamma_{i,2/1}).
\end{gathered}\end{equation*}
By Lemma~\ref{lem:transfer-properties}, properties \ref{itm:morphisms} and \ref{itm:sums-and-products}, we then have
$\tr_{\Gamma,\sigma}^{G,\beta_i} Y_i=\sum_{j=0}^2 \vol\big(X_{i,j}\big)\cdot{\varphi_j}^*$,
where $\varphi_j\colon G\to\Gamma$ denotes the group homomorphism that describes the $G$--action on the $j$--th summand.
On the summands $\Gamma\times X_{i,1}$ and $\Gamma\times X_{i,2}$ the $G$--action factors through the group $\Z$, which has trivial bounded cohomology in positive degrees.
On $\Gamma\times X_{i,0}$ we have $\varphi_0=q_\Gamma\circ s$, where $q_\Gamma\colon\pi_1(M)\to\Gamma$ is the quotient homomorphism.
In summary, we get that
$\tr_{\Gamma,\sigma}^{G,\beta_i}Y_i=\phi\vcentcolon=\kappa\cdot (q_\Gamma\circ s)^*$ on $\EHbr^{>0}$.
Because $q_\Gamma\circ s$ has a left-inverse, induced from $r$, the map $\phi$ is surjective onto $\EHbr^*(\frgrp_2)$.

The statement for $d=3$ now follows from Lemma~\ref{lem:higher-degrees}, part~\ref{itm:direct-calculation} and the fact that $\EHbr^3(\frgrp_2)$ is infinite-dimensional \cite{soma}. For $d\geq 2$ even, it follows from Lemma~\ref{lem:higher-degrees}, part~\ref{itm:pairing-calculation} because
$\Hb^2(\frgrp_2)$ is infinite-dimensional \cite{brooks}, and because
the Kronecker map $\mathit{Kr}$ sets up an isomorphism $\EHb^2(\frgrp_2)=\EHbr^2(\frgrp_2)\cong\big(\Hlr_2(\frgrp_2)\big)'=\big(\Hl_2(\frgrp_2)\big)'$ by
\cite{matsumoto-morita}*{Corollary~2.7 and Theorem~2.3}.
\end{proof}

\begin{theorem}\label{thm:disk-bounded}
Let $\Diff_\vol(\disk,\partial\disk)$ denote the group of volume-preserving (smooth) diffeomorphisms of the standard $2$--disk that restrict to the identity in a neighborhood of the boundary.
Then $\dim\EHbr^d(\Diff_\vol(\disk,\partial\disk))=\infty$ for $d=3$ and for $d\geq 2$ even.
\end{theorem}
\begin{proof}
Let $\Gamma$ be the group of pure braids on three strands, and $\Lambda\vcentcolon=\Diff_\vol(\disk,\partial\disk)$ with the multiplication $\lambda_1\lambda_2=\lambda_2\circ\lambda_1$. We also fix some $n\in\N$ such that $2n\geq d$.

\textbf{1)}
Let $C_3=\bigl\{(x_1,x_2,x_3)\in(\disk)^3\bigm| x_i\neq x_j\enspace\forall i\neq j\bigr\}$ be the configuration space of three ordered points in the (closed) disk, and let $\overline{C}_3$ be its Fulton--MacPherson compactification, i.e., the closure of the image of the function
\begin{gather*}
\big(\id,\,(\tau_{i,j})_{1\leq i\neq j\leq 3},\,(\theta_{i,j,k})_{\{i,j,k\}=\{1,2,3\}}\big)\colon\, C_3\to\R^6\times (S^1)^6\times [0,\infty]^6,\\
\text{where}\qquad
\tau_{i,j}(x_1,x_2,x_3)=\frac{x_i-x_j}{\|x_i-x_j\|},\qquad\theta_{i,j,k}(x_1,x_2,x_3)=\frac{\|x_i-x_j\|}{\|x_j-x_k\|}.
\end{gather*}
By \cite{sinha}, the space $\overline{C}_3$ is a topological manifold with boundary and the inclusion $C_3\to\overline{C}_3$ induces an isomorphism of fundamental groups $\pi_1(\overline{C}_3)\cong\pi_1(C_3)\cong\Gamma$.
We let the measure space of the required coupling $(\Gamma,\sigma,X,\rho,\Gamma)$ be the universal covering of $\overline{C}_3$. It carries the measure induced from $\disk$ and the left action $\sigma\colon\Gamma\curvearrowright X$ of deck transformations. The bounded subsets are the precompact sets.

We define the right action $\rho\colon X\curvearrowleft\Lambda$ via monodromy:
By \cite{smale}, the space
$\Diff(\disk,\partial\disk)$ is contractible with regard to the $C^r$--topology.
For every $\lambda\in\Diff_\vol(\disk,\partial\disk)$ we pick a path $H$ from the identity to $\lambda$ and let $\widehat{H}$ be the induced isotopy on $C_3$. Via the differential, $\widehat{H}$ extends to an isotopy on $\overline{C}_3$, which in turn lifts to an isotopy $\widetilde{H}$ on $X$. We define $x.\lambda=\widetilde{H}_1(x)$.
Because $\Diff(\disk,\partial\disk)$ is simply connected, the definition does not depend on the choice of $H$.

\textbf{2)}
We pick $2n$--many subsets $\{N(x_{i,1/2})\}_{i\leq n}$ of $\disk$, each of them in the shape of a small disk with center $x_{i,1/2}$ and with the same small radius $\epsilon$. We do this in such a way that the unions $N(x_{i,1})\cup N(x_{i,2})$ are pairwise disjoint. Furthermore, we can arrange that the sets $N(x_{i,1})\setminus N(x_{i,2})$, $N(x_{i,2})\setminus N(x_{i,1})$ and $N(x_{i,1})\cap N(x_{i,2})$ all have the same positive volume $\kappa$, which is also independent of $i$.

Let $G\vcentcolon=\frgrp_2=\langle a_1,a_2\rangle$ and define the homomorphism sequences $\big(\alpha_{i,l}\colon G\to\Lambda\big)_{i\leq n,l\in\N}$ by the actions of the two generators: For $j\in\{1,2\}$ we let the map $\alpha_{i,l}(a_j)\colon X\to X$, be the identity map outside of $N(x_{i,j})$. Inside $N(x_{i,j})$ we let it be the end result of the ``twisting'' homotopy that fixes the boundary of the $\epsilon$--disk, rotates the $(\epsilon-\epsilon/l)$--disk centered in $x_{i,j}$ around once, and interpolates smoothly on the rest.
The sequence $(\rho\circ\alpha_{i,l})_{l\in\N}$ converges, in the sense of
Lemma~\ref{lem:transfer-properties}, property~\ref{itm:limits}, to the action $\beta_i$, where $\beta_i(a_j)$ is induced by rotating the whole subset $N(x_{i,j})$ and fixing its complement.

\textbf{3)}
For each $i\leq n$, the configuration space $C_3$ decomposes into a disjoint union of $4^3$--many pieces, according to which of the three configuration-coordinates lie in $N(x_{i,1})$ and which lie in $N(x_{i,2})$. This decomposition lifts to decompositions $\overline{C}_3=\bigsqcup_\pi X_\pi$ and $X\cong\bigsqcup_\pi \Gamma\times X_\pi$. If $\varphi_\pi\colon G\to\Gamma$ denotes the homomorphism that describes the action $\beta_i$ on the $\pi$--th summand, then $\varphi_\pi$ is non-trivial only on those summands for which at least two of the configuration-coordinates lie in $N(x_{i,1})\cup N(x_{i,2})$. This means that the supports $Y_i$ of the $\beta_i$ are disjoint for varying $i$.

To compute $\tr_{\Gamma,\sigma}^{G,\beta_i}Y_i=\sum_\pi\vol(X_\pi)\cdot {\varphi_\pi}^*$ we first recall that the pure braid group $G$ is isomorphic to $\Z\times\frgrp_2=\langle b_0\rangle\times\langle b_{1},b_{2}\rangle$. Here, the generator $b_0$ of the center rotates all three strands, and the other two generators rotate the third strand around the first, respectively the second strand.
On those summands $X_\pi$ where one set $N(x_{i,1/2})$ contains all three configuration-coordinates, the images $\varphi_\pi(a_1), \varphi_\pi(a_2)$ commute, such that $\varphi_\pi$ factors through the abelian group $\Z^2$, and hence ${\varphi_\pi}^*=0$ on $\EHbr^{>0}$. The same happens when one set $N(x_{i,1/2})$ contains only one or none of the coordinates, and also when one coordinate is neither contained in $N(x_{i,1})$ nor $N(x_{i,2})$. There remain six summands for the cases where each of the sets $N(x_{i,1})\setminus N(x_{i,2})$, $N(x_{i,2})\setminus N(x_{i,1})$, $N(x_{i,1})\cap N(x_{i,2})$ contain exactly one coordinate. Here, $\varphi_\pi(a_{1})$ and $\varphi_\pi(a_{2})$ rotate the strand corresponding to the coordinate that lies in $N(x_{i,1})\cap N(x_{i,2})$ around each of the other two strands.

Let $\varphi\colon G\to\Gamma$ be the homomorphism that sends $a_j$ to $b_j$.
For each of the above six summands the map $\varphi_\pi$ differs from $\varphi$ only by the ordering of the strands. Then, if $\iota\colon G\to B_3$ denotes the inclusion into the full braid group, $\iota\circ\varphi_\pi$ and $\iota\circ\varphi$ are related by an inner automorphism of $B_3$. Inner automorphisms induce the identity on bounded cohomology, and $\iota^*$ is an isomorphism because $\iota$ is the inclusion of a finite-index subgroup. Therefore, ${\varphi_\pi}^*={\varphi}^*$, and finally we get $\tr_{\Gamma,\sigma}^{G,\beta_i}Y_i=\phi\vcentcolon=6\kappa^3\cdot {\varphi}^*$ on $\EHbr^{>0}$.

Because the group homomorphism $\varphi$ has a left-inverse, $\phi$ is surjective.
As in Theorem~\ref{thm:M-bounded}, the result now follows from Lemma~\ref{lem:higher-degrees}, together with the fact that both $\EHbr^3(\frgrp_2)$ and $\mathit{Kr}\big(\EHbr^2(\frgrp_2)\big)$ are infinite-dimensional.
\end{proof}

\begin{remark}
To extend Theorems~\ref{thm:M-bounded} and \ref{thm:disk-bounded} to bounded cohomology in odd degrees $d\geq 5$, it would suffice to show that $\Hlr_3(\frgrp_2)\neq \{0\}$.
\end{remark}

Finally, the proof of Theorem~\ref{thm:M-bounded} can also be adapted to ordinary cohomology:

\begin{theorem}\label{thm:M-ordinary}
Let $M$ be a compact Riemannian manifold of dimension $\geq 5$ and assume that there exists a split epimorphism $r\colon\pi_1(M)\to\Z^2$ that is trivial on the center $Z(\pi_1(M))$.
Then $\Ho^d(\Homeo_{\vol,0}(M))\neq \{0\}$ for all $d\geq 0$.
\end{theorem}

\begin{proof}
As in Theorem~\ref{thm:M-bounded}, let $\Gamma\vcentcolon=\pi_1(M)/Z(\pi_1(M))$,
$\Lambda\vcentcolon=\Homeo_{\vol,0}(M)$, and let $n\in\N$ such that $2n\geq d$.

\textbf{1)} We use the same coupling $(\Gamma,\sigma,X,\rho,\Lambda)$ as in the proof of Theorem~\ref{thm:M-bounded}.

\textbf{2)}
Let $G=\Z^2$ with generators $a_1,a_2$ and let $s\colon G\to\pi_1(M)$ be a split of $r$. We construct a map $f\colon\mathbb{T}^2\to M$ from the $2$--torus such that $\pi_1(f)=s$: On the $1$--skeleton of the usual CW-structure, $S^1\vee S^1\subset\mathbb{T}^2$, the map is given by two loops in $M$ that represent the elements $s(a_1), s(a_2)$, and it is extended to $\mathbb{T}^2$ by the homotopy that proves that these elements commute. Because $\dim\mathbb{T}^2$ is below the middle dimension of $M$, we can use Whitney's Embedding Theorem \cite{whitney} to homotope $\sqcup f\colon\bigsqcup_{i=1}^n\mathbb{T}^2\to M$ to a differentiable embedding of $n$--many disjoint tori into $M$. We pick disjoint tubular neighborhoods $\{N_i\}_{i\leq n}$ of these tori, all with the same volume $\kappa$ and identify each of them them, by a volume-preserving homeomorphism, with $S^1\times S^1\times B_\epsilon$, the product of two standard circles and a Euclidean ball.

As before, we define the sequences of homomorphisms $\big(\alpha_{i,l}\colon G=\Z^2\to\Lambda\big)_{i\leq n,l\in\N}$ by the actions of the two generators: 
The map $\alpha_{i,l}(a_j)\colon X\to X$ is the identity outside of $N_i$. Inside $N_i\cong S^1\times S^1\times B_\epsilon$ it is the end result of the homotopy that fixes the boundary, pushes $S^1\times S^1\times B_{\epsilon-\epsilon/l}$ once around the $j$--th $S^1$--factor, and uses the norm of the $B_\epsilon$--coordinate as a parameter to interpolate on the rest.
The sequence $(\rho\circ\alpha_{i,l})_{l\in\N}$ converges, in the sense of Lemma~\ref{lem:transfer-properties}, property~\ref{itm:limits}. In the limit $\beta_i$ the action of $a_j$ is induced by rotating the whole subset $N_i$ around the $j$--th $S^1$--factor and fixing its complement.

This time, we also have to check that condition \eqref{eqn:bounded-convergence} is satisfied. It suffices to do this on the generators $a_j$. If $\Phi_{i,j}\colon [0,1]\times p^{-1}(N_i)\to p^{-1}(N_i)$ denotes the homotopy induced from rotating $N_i$ around the $j$--th $S^1$--factor, then for all $i,j$ and any fundamental domain $F$ the set
$\bigcup_{l\in\N}\rho\circ\alpha_{i,l}(a_j)(F)\subset F\cup\Phi_{i,j}\Big([0,1]\times\big(F\cap p^{-1}(N_i)\big)\Big)$ is precompact.

\textbf{3)}
For each $i$ the limit action $\beta_i$ has support
$Y_i\vcentcolon=p^{-1}(N_i)$. The coupling $(\Gamma,\sigma,Y_i,\beta_i,G)$ decomposes as $Y_i\cong\Gamma\times N_i$, and the $G$\nobreakdash--action is given by the group homomorphism $q_\Gamma\circ s$.
As before, it follows that $\tr_{\Gamma,\sigma}^{G,\beta_i}Y_i=\phi\vcentcolon=\kappa\cdot (q_\Gamma\circ s)^*$.
Because the composition $\phi\circ r^*=\kappa\cdot\id$ is an isomorphism, $\phi$ is surjective onto $\Ho^*(G)$ and
$\mathit{Kr}\circ\phi\big(\Ho^*(\Gamma)\big)\neq \{0\}$ in degree $\leq 2$.
The statement now follows from Lemma~\ref{lem:higher-degrees}, part~\ref{itm:pairing-calculation}.
\end{proof}

\section*{Acknowledgments}

This research was supported by the ERC Consolidator Grant No.\ 681207.

The author thanks Michał Marcinkowski for helpful discussions.


\begin{bibdiv}
\begin{biblist}

\bib{banyaga}{article}{
   author={Banyaga, Augustin},
   title={Formes-volume sur les vari\'{e}t\'{e}s \`a bord},
   language={French},
   journal={Enseign. Math. (2)},
   volume={20},
   date={1974},
   pages={127--131},
}

\bib{brandenbursky-marcinkowski}{article}{
   author={Brandenbursky, Michael},
   author={Marcinkowski, Michał},
   title={Bounded cohomology of transformation groups},
   note={arXiv:1902.11067 [math.GT]},
   status={preprint},
}

\bib{brooks}{article}{
   author={Brooks, Robert},
   title={Some remarks on bounded cohomology},
   conference={
      title={Riemann surfaces and related topics: Proceedings of the 1978
      Stony Brook Conference},
      address={State Univ. New York, Stony Brook, N.Y.},
      date={1978},
   },
   book={
      series={Ann. of Math. Stud.},
      volume={97},
      publisher={Princeton Univ. Press, Princeton, N.J.},
   },
   date={1981},
   pages={53--63},
}

\bib{gambaudo-ghys}{article}{
   author={Gambaudo, Jean-Marc},
   author={Ghys, \'{E}tienne},
   title={Commutators and diffeomorphisms of surfaces},
   journal={Ergodic Theory Dynam. Systems},
   volume={24},
   date={2004},
   number={5},
   pages={1591--1617},
}

\bib{hatcher}{book}{
   author={Hatcher, Allen},
   title={Algebraic topology},
   publisher={Cambridge University Press, Cambridge},
   date={2002},
   pages={xii+544},
   isbn={0-521-79160-X},
   isbn={0-521-79540-0},
}


\bib{kimura}{article}{
   author={Kimura, Mitsuaki},
   title={Gambaudo--Ghys construction on bounded cohomology},
   note={arXiv:2009.00124 [math.GT]},
   status={preprint},
}

\bib{loeh}{article}{
   author={L\"{o}h, Clara},
   title={A note on bounded-cohomological dimension of discrete groups},
   journal={J. Math. Soc. Japan},
   volume={69},
   date={2017},
   number={2},
   pages={715--734},
}

\bib{matsumoto-morita}{article}{
   author={Matsumoto, Shigenori},
   author={Morita, Shigeyuki},
   title={Bounded cohomology of certain groups of homeomorphisms},
   journal={Proc. Amer. Math. Soc.},
   volume={94},
   date={1985},
   number={3},
   pages={539--544},
}


\bib{monod-shalom}{article}{
   author={Monod, Nicolas},
   author={Shalom, Yehuda},
   title={Orbit equivalence rigidity and bounded cohomology},
   journal={Ann. of Math. (2)},
   volume={164},
   date={2006},
   number={3},
   pages={825--878},
}

\bib{savini}{article}{
  author={Savini, Alessio},
  title={Algebraic hull of maximal measurable cocycles of surface groups into Hermitian Lie groups},
  journal={Geom Dedicata},
  date={2020},
}

\bib{shalom}{article}{
   author={Shalom, Yehuda},
   title={Harmonic analysis, cohomology, and the large-scale geometry of
   amenable groups},
   journal={Acta Math.},
   volume={192},
   date={2004},
   number={2},
   pages={119--185},
}

\bib{sinha}{article}{
   author={Sinha, Dev P.},
   title={Manifold-theoretic compactifications of configuration spaces},
   journal={Selecta Math. (N.S.)},
   volume={10},
   date={2004},
   number={3},
   pages={391--428},
}

\bib{smale}{article}{
   author={Smale, Stephen},
   title={Diffeomorphisms of the $2$-sphere},
   journal={Proc. Amer. Math. Soc.},
   volume={10},
   date={1959},
   pages={621--626},
}

\bib{soma}{article}{
   author={Soma, Teruhiko},
   title={Bounded cohomology and topologically tame Kleinian groups},
   journal={Duke Math. J.},
   volume={88},
   date={1997},
   number={2},
   pages={357--370},
}

\bib{whitney}{article}{
   author={Whitney, Hassler},
   title={Differentiable manifolds},
   journal={Ann. of Math. (2)},
   volume={37},
   date={1936},
   number={3},
   pages={645--680},
}

\end{biblist}
\end{bibdiv}
\end{document}